\documentclass[12pt]{article}
\usepackage[english]{babel}
\usepackage{amsmath}
\usepackage{amsthm}
\usepackage{amssymb}
\usepackage{mathtools}
\usepackage{enumitem}
\usepackage{hyperref}
\usepackage{color}
\newtheorem{theo}{Theorem}
\newtheorem{defi}{Definition}
\newtheorem{lemma}{Lemma}

\newtheorem{cor}{Corollary}

\newtheorem{ques}{Question}

\newcommand{\N}{\mathbb{N}}

\title{Ramsey upper density of infinite graph factors}
\date{\vspace{-5ex}}
\author{
    J\'ozsef Balogh \thanks{Department of Mathematics, University of Illinois at Urbana-Champaign, Urbana, Illinois 61801, USA; and Moscow Institute of Physics and Technology, Russian Federation. Research is partially supported by NSF Grant DMS-1764123, Arnold O. Beckman Research Award (UIUC Campus Research Board RB 18132), the Langan Scholar Fund (UIUC), and the Simons Fellowship.}

    \and 
	Ander Lamaison\thanks{Faculty of Informatics, Marasyk University, Brno, Czech Republic. {\tt lamaison@fi.muni.cz}. Previous affiliation: Freie Universit\"at Berlin, Berlin, Germany. Funded by the Deutsche Forschungsgemeinschaft (DFG, German Research Foundation) under Germany's Excellence Strategy - The Berlin Mathematics Research Center MATH+ (EXC-2046/1, project ID: 390685689).}
	
	}

\begin{document}
\maketitle

\begin{abstract}
    The study of upper density problems on Ramsey theory was initiated by Erd\H{o}s and Galvin in 1993. In this paper we are concerned with the following problem: given a fixed finite graph $F$, what is the largest value of $\lambda$ such that every 2-edge-coloring of the complete graph on $\N$ contains a monochromatic infinite $F$-factor whose vertex set has upper density at least $\lambda$?
    
    Here we prove a new lower bound for this problem. For some choices of $F$, including cliques and odd cycles, this new bound is sharp, as it matches an older upper bound. For the particular case where $F$ is a triangle, we also give an explicit lower bound of $1-\frac{1}{\sqrt{7}}=0.62203\dots$, improving the previous best bound of 3/5. 
\end{abstract}

\section{Introduction}

Let $K_\N$ denote the complete graph on the natural numbers. Given a subset $S\subseteq\N$, its upper density is defined as $\bar d(S)=\limsup\limits_{t\rightarrow\infty}\frac{|S\cap[t]|}{t}$. Given a subgraph $G\subseteq K_\N$, its upper density is $\bar d(G)=\bar d(V(G))$.

\begin{defi}
Let $H$ be a countably infinite graph. The Ramsey (upper) density of $H$, denoted by $\rho(H)$, is the supremum of the values of $\lambda$ such that, in every red-blue-coloring of the edges of $K_\N$, there exists a monochromatic subgraph $H'\subseteq K_\N$ isomorphic to $H$ with $\bar d(H')\geq\lambda$.
\end{defi}

This concept was first introduced by Erd\H{o}s and Galvin \cite{ErdGal}, for the particular case of the infinite path $P_\infty$. They proved that $2/3\leq\rho(P_\infty)\leq 8/9$.  After some improvements on the lower bound by DeBiasio and McKenney~\cite{BiaKen} and by Lo, Sanhueza-Matamala and Wang~\cite{LoSanWan}, the Ramsey density of the infinite path was determined by Corsten, DeBiasio, Lang and the second author in~\cite{CDLL}, as $\rho(P_\infty)=(12+\sqrt{8})/17= 0.87226\dots$ The parameter for general $H$ was first defined by DeBiasio and McKenney~\cite{BiaKen}. Many results on the Ramsey density of different classes of graphs $H$ can be found in \cite{Lam} and \cite{CorBiaKen}.

We focus on the case of infinite factors. Given a finite graph $F$, the infinite $F$-factor, denoted by $\omega\cdot F$, is the graph formed by infinitely many pairwise disjoint copies of $F$. Several results relating to the Ramsey density of infinite factors appear in \cite{Lam}. 

Some of these are given in terms of a certain fixed function $f$ which is defined as follows. The reader should be forgiven for skipping such an unwieldy definition, as it will not be actually used in this paper. Instead, we will rely on the results from \cite{Lam} that use $f$. These will be treated as black boxes.

\begin{defi} Given a continuous function $g:[0,+\infty)\rightarrow\mathbb{R}$ and $\gamma,t\in\mathbb{R}$, we define \[\Gamma^+_\gamma(g,t)=\min\{x:\gamma x+g(x)\geq t\},\hskip 0.8cm \Gamma^-_\gamma(g,t)=\min\{x:\gamma x-g(x)\geq t\},\] where the minimum of the empty set is taken as $+\infty$. We define the function $h(\gamma)$ as the infimum, over all 1-Lipschitz functions $g$ with $g(0)=0$, of \[\limsup\limits_{t\rightarrow\infty}\frac{\Gamma^+_\gamma(g,t)+\Gamma^-_\gamma(g,t)}{t}.\] We take $f:(0,+\infty)\rightarrow\mathbb{R}$ as \[f(\lambda)=1-\frac{1}{\frac{2\lambda}{(1+\lambda)^2}h\left(\frac{\lambda-1}{\lambda+1}\right)+\frac{2\lambda}{1+\lambda}}.\] We extend this definition by continuity to $f(0)=1$ and $f(+\infty)=1/2$.\end{defi}

\begin{figure}
\begin{centering}
\includegraphics[width=60mm]{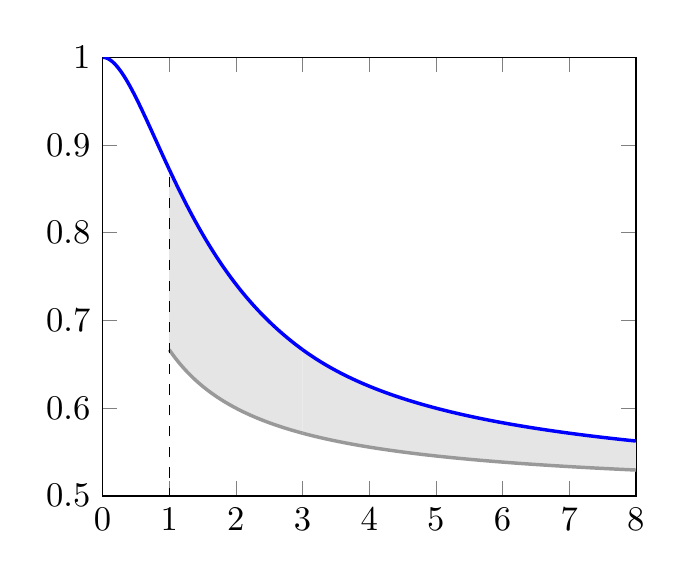}
\caption{Plot of the bounds on the function $f(x)$. The conjectured value is given in blue.}
\label{nbex}
\par\end{centering}
\end{figure}

This function satisfies the following bounds: \begin{equation}\label{func}\frac{x+1}{2x+1}\leq f(x) \leq \left\{
  \begin{array}{ccc}
    \frac{2x^2+3x+7+2\sqrt{x+1}}{4x^2+4x+9} & \text{for} & 0\leq x< 3,\\
    & & \\
    \frac{x+1}{2x} & \text{for} & x \geq 3,
  \end{array}
\right.\end{equation}
and the upper bound is known to be sharp for $x\in[0,1]$. A conjecture in~\cite{Lam} states that the upper bound is the correct value for all $x$. Note that, with this definition, $\rho(P_\infty)=f(1)$. The three results from \cite{Lam} that apply to infinite factors, one upper bound and two lower bounds, are as follows:

\begin{theo}[\cite{Lam}]\label{previous}
\begin{enumerate}[label=(\roman*)]
    \item\label{oldupp} For every non-empty finite graph $F$, we have \[\rho(\omega\cdot F)\leq f\left(\min\limits_{\substack{I\text{ indep. in }F\\I\neq\emptyset}}\frac{|N(I)|}{|I|}\right).\]
    \item\label{oldlow} Let $F$ be a graph, and let $I\subseteq V(F)$ be a non-empty independent set such that $N(I)$ is also independent. Then $\rho(\omega\cdot F)\geq f\left(\frac{|N(I)|}{|I|}\right)$.
    \item\label{oldbes} Let $F$ be a finite graph. Then $\rho(\omega\cdot F)\geq \frac{|V(F)|}{2|V(F)|-\alpha(F)}$.
\end{enumerate}
\end{theo}

For certain choices of $F$, these bounds are sufficient to find the exact value of the Ramsey density. Indeed, if one of the independent sets $I$ that minimize the ratio $\frac{|N(I)|}{|I|}$ also satisfies that $N(I)$ is independent, then \ref{oldupp} and \ref{oldlow} combined give that $\rho(\omega\cdot F)=\frac{|N(I)|}{|I|}$. This is the case, for example, for all bipartite graphs.

The bound \ref{oldbes} is derived from a result of Burr, Erd\H{o}s and Spencer \cite{BurErdSpe}, and does not depend on the function $f$. There is currently no known graph $F$ for which this bound is tight and, in fact, in this paper we will show that, unless the lower bound of \eqref{func} is tight for some $x$, the bound \ref{oldbes} is never tight.

Two interesting cases in which the exact value does not follow from these bounds are cliques and odd cycles. 

\begin{itemize}\item In the case of cliques, we have $\rho(\omega\cdot K_n)\leq f(n-1)\leq \frac{n}{2n-2}$ from \ref{oldupp} and $\rho(\omega\cdot K_n)\geq \frac{n}{2n-1}$ from \ref{oldbes}. We cannot obtain any lower bound from \ref{oldlow} because there is no non-empty independent set $I$ such that $N(I)$ is independent, if $n\geq 3.$
\item In the case of odd cycles, we have $\rho(\omega\cdot C_{2n+1})\leq f\left(\frac{n+1}{n}\right)$ from \ref{oldupp}, $f(\omega\cdot C_{2n+1})\geq f\left(\frac{n}{n-1}\right)$ from \ref{oldlow} (if $n\geq 2$), and $\rho(\omega\cdot C_{2n+1})\geq \frac{2n+1}{3n+2}$ from \ref{oldbes}. For $n$ large enough, the lower bound from \ref{oldlow} is better than that from \ref{oldbes}, although the latter has the advantage of being easier to compute.\end{itemize}

In this paper we will prove the following new lower bound for infinite factors:

\begin{theo}\label{mainfac} Let $F$ be a finite graph. Then $\rho(\omega\cdot F)\geq f\left(\frac{|V(F)|}{\alpha(F)}-1\right)$.\end{theo}

This bound is not weaker than Theorem \ref{previous}\ref{oldbes}. That result can be expressed as $\rho(\omega\cdot F)\geq\bar f\left(\frac{|V(F)|}{\alpha(F)}-1\right)$, for $\bar f(x)=\frac{x+1}{2x+1}$. By \eqref{func}, we have $f(x)\geq \bar f(x)$.

We can compare this lower bound to the upper bound in Theorem \ref{previous}\ref{oldupp} to find a new family of graphs for which we can obtain an exact result:

\begin{cor}\label{corclicyc} Let $F$ be a finite graph. Suppose that, among the independent sets $I$ that minimize $\frac{|N(I)|}{|I|}$, there is at least one with size $\alpha(F)$. Then $\rho(\omega\cdot F)= f\left(\frac{|V(F)|}{\alpha(F)}-1\right)$. In particular:
\begin{itemize}
    \item $\rho(\omega\cdot K_n)=f(n-1)$ for all $n\geq 1$.
    \item $\rho(\omega\cdot C_{2n+1})=f\left(\frac{n+1}{n}\right)$ for all $n\geq 1$.
\end{itemize}\end{cor}

Since a triangle is both a clique and an odd cycle, these results show that $\rho(\omega\cdot K_3)=f(2)$, answering a question from \cite{Lam}. Note however that, since we do not know the actual value of $f(2)$, the explicit bounds on $\rho(\omega\cdot K_3)$ have not improved: all we know is that $\frac35\leq \rho(\omega\cdot K_3)\leq (21+\sqrt{12})/33= 0.74133\dots$ (we suspect that the upper bound is sharp). The next result improves upon this lower bound:

\begin{theo}\label{triaex}$\rho(\omega\cdot K_3)\geq 1-\frac{1}{\sqrt{7}}\approx 0.62204$ (and therefore $f(2)\geq 1-\frac{1}{\sqrt{7}}$).\end{theo}

Unlike Theorem \ref{mainfac}, whose proof is based on the techniques from \cite{CDLL} and \cite{Lam} (which is where the function $f$ comes from), Theorem \ref{triaex} is based on a careful analysis of the technique of Burr, Erd\H{o}s and Spencer \cite{BurErdSpe}.

\section{Proof of Theorem \ref{mainfac}}

The proof of Lemma \ref{mainfac} uses the same overarching outline as Theorem \ref{previous}\ref{oldlow}: given an edge coloring of $K_\N$ we color the vertices of the graph while encoding some information, we find a smaller substructure in the resulting coloring and finally use the information in the vertices to extend this substructure into our desired graph.

The main difference comes from the choice of coloring. While the proof from \cite{Lam} uses a coloring due to Elekes, Soukup, Soukup and Szentmikl\'ossy~\cite{ESSS}, here instead we will use a coloring that arises from a straightforward application of Ramsey's theorem. The reason for this change is that the coloring from \cite{ESSS} is good at finding ``connections at infinity", allowing one to add infinitely many vertices to a component with little restriction as to how they are joined. Because the components of $\omega\cdot F$ all have bounded size, we do not need to worry about connecting more and more vertices into a component. Instead, we will be able to create each component of $\omega\cdot F$ in one or two steps.

The intermediate step of finding a smaller substructure is essentially the same. We will use the following lemma:

\begin{lemma}\label{blackbox2} Let $r$ and $s$ be positive integers. For every $\epsilon>0$ there exist $\tau>0$ and $N$ with the following property: for every graph $G$ on vertex set $[n]$, with $n>N$, such that $\delta(G)>(1-\tau)n$, in every coloring $\Psi:V(G)\cup E(G)\rightarrow\{R,B\}$ there exists a subgraph $F\subseteq G$, a color $C\in\{R,B\}$ and a value $k\in[\delta n,n]$ such that $|V(F)\cap[k]|\geq (f(s/r)-\epsilon)k$ and every component of $F$ is of one of the following types:\begin{itemize}
\item An isolated vertex of color $C$.
\item A copy of $K_{r,s}$ in which the edges and the side of size $s$ have color $C$, and the side of size $r$ has the opposite color.
\end{itemize}
\end{lemma}

This is a finitary version of Lemma 2 in \cite{Lam}. The same proof that was used in that paper works here.

The next lemma will require the following definition, which will be the key information that will be encoded into our vertex coloring:

\begin{defi} Given a set $X$ of vertices in an edge-colored graph, a real number $\epsilon>0$, a natural number $s$ and a color $C\in\{R,B\}$, we say that $X$ is $(C,\epsilon, s)$-adequate if every subset of $X$ of size at least $\epsilon|X|$ contains a $C$-colored clique of size $s$. \end{defi}

Observe that in particular, if a set $X$ is $(C,\epsilon, s)$-adequate, then it contains a disjoint family of $C$-colored copies of $K_s$ covering at least $(1-\epsilon)|X|$ vertices. This is proved by considering a maximal family of such cliques.

\begin{lemma}\label{refic} For every $\epsilon>0$ and $s\in\N$ there exist $T,N\in\N$ with the following property: for every $\Psi:E(K_n)\rightarrow \{R,B\}$ with $n>N$ there exists a partition $V(K_n)=V_1\cup V_2\cup\dots\cup V_T$ into almost equal parts, and colors $C_1, \dots, C_T\in \{R,B\}$ in which all but $\epsilon T$ sets $V_i$ are $(C_i,\epsilon, s)$-adequate.\end{lemma}

\begin{proof} Let $\alpha=\lceil2\epsilon^{-1}\rceil s$ and $q=R(K_\alpha, K_\alpha)$. Take a maximal family $\mathcal{F}$ of disjoint monochromatic copies of $K_\alpha$ ($\mathcal{F}$ might contain cliques of different colors). The vertex set $V(K_n)\setminus V(\mathcal{F})$ does not contain any monochromatic $K_\alpha$, so its size is at most $q$. 

Let $T=\lceil 3\epsilon^{-1}\rceil$. Each of the sets $V_1, \dots, V_T$ will have size either $\lfloor n/T\rfloor$ or $\lceil n/T\rceil$. Let $\beta=\lfloor \lfloor n/T\rfloor/\alpha\rfloor$, this is the number of $\alpha$-cliques that fit into a set $V_i$ (of the smaller type). Let $\mathcal{F}_R$ and $\mathcal{F}_B$ be the sets of red and blue cliques from $\mathcal{F}$, respectively.

Create $V_1, V_2, \dots, V_T$, all initially empty. For $\lfloor |\mathcal{F}_R|/\beta\rfloor+\lfloor |\mathcal{F}_B|/\beta\rfloor$ of these sets, put $\beta$ cliques from $\mathcal{F}$ of the same color into each. We call these sets pseudo-adequate, and we will later show that they are indeed adequate. Distribute the remaining vertices from $K_n$ into the sets $V_i$ so that the resulting sets are almost equal. The number of pseudo-adequate sets is \[\left\lfloor \frac{|\mathcal{F}_R|}\beta\right\rfloor+\left\lfloor \frac{|\mathcal{F}_B|}\beta\right\rfloor\geq\left\lfloor\frac{|\mathcal{F}|}{\beta}\right\rfloor-1\geq \frac{\mathcal{F}}{\beta}-2\geq \frac{\frac{n-q}{\alpha}}{\beta}-2\geq T-\frac{q}{\alpha\beta}-2\geq T-3\] if $\beta\geq q$ (this will be the case if $n>q\alpha T$). That means that there are at most three sets which are not pseudo-adequate, which represents less than $\epsilon T$ of the sets.

For each pseudo-adequate set $V_i$, we let $C_i$ be the color of the cliques from $\mathcal{F}$ that were put into it. For the (up to three) remaining sets $V_i$, choose $C_i$ arbitrarily. If $V_i$ is pseudo-adequate, there are at most $\alpha$ vertices not in a copy of $K_\alpha$ from $\mathcal{F}$. 

In a pseudo-adequate set $V_i$, take a subset $S\subseteq V_i$ of size at least $\epsilon|V_i|$. There are at least $\epsilon|V_i|-\alpha$ vertices in $S$ which belong to a clique in $\mathcal{F}$, so by the pigeonhole principle $S$ contains at least $\lceil\frac{\epsilon|V_i|-\alpha}{\beta}\rceil$ elements from some clique in $\mathcal{F}$. This is at least \[\frac{\epsilon|V_i|-\alpha}{\beta}\geq\frac{\epsilon\lfloor\frac{n}{T}\rfloor-\alpha}{\left\lfloor\frac{\left\lfloor\frac nT\right\rfloor}{\alpha}\right\rfloor}\geq \frac{\frac\epsilon2\lfloor\frac nT\rfloor}{\frac{\left\lfloor\frac nT\right\rfloor}\alpha}=\frac{\alpha\epsilon}{2}\geq s\] if $\alpha<\frac{\epsilon}{2}\left\lfloor\frac nT\right\rfloor$. This holds if $n>\lceil2\epsilon^{-1}\alpha\rceil T$. We conclude that we can take $N=\max\{q\alpha T,  \lceil2\epsilon^{-1}\alpha\rceil T\}$.\end{proof}

We will prove a finitary version of Theorem~\ref{mainfac}, from which the infinite version follows directly:

\begin{lemma}\label{maintwo} For every finite graph $F$ and every $\epsilon>0$ there exists $N$ and $\tau>0$ with the following property: for every $n>N$, for every coloring $\Psi:E(K_n)\rightarrow \{R,B\}$ there exists $t\in[\tau n,n]$ and a monochromatic family $\mathcal{F}$ of disjoint copies of $F$ such that \[\frac{|V(\mathcal{F})\cap[t]|}{t}\geq f\left(\frac{|V(F)|}{\alpha(F)}-1\right)-\epsilon.\]\end{lemma}

\begin{proof}[Proof of Theorem \ref{mainfac}] Fix $\epsilon>0$. Let $\Psi:E(K_\N)\rightarrow\{R,B\}$ be an edge-coloring. Consider a sequence $n_1<n_2< \dots$ of natural numbers with $\lim n_{i+1}/n_i\rightarrow\infty$. For each $i\geq 2$, consider the colored clique $G_i$ on vertex set $(n_{i-1},n_i]$, whose edge-coloring is $\Psi_i$, the one induced by $\Psi$. Now consider the graph $G_i'$, obtained by substracting $n_{i-1}$ from the label on each vertex.

By Lemma \ref{maintwo}, there exists a fixed $\tau>0$ such that, for every $i$ large enough (enough to have $n_i-n_{i-1}>N$) there is a value $t_i\in[\tau (n_i-n_{i-1}),n_i-n_{i-1}]$, and a monochromatic family $\mathcal{F}'_i$ of disjoint copies of $F$ in $G'_i$ with color $C_i$, such that \[\frac{|V(\mathcal{F}'_i)\cap[t_i]|}{t_i}\geq f\left(\frac{|V(F)|}{\alpha(F)}-1\right)-\epsilon.\]

For infinitely many values of $i$, the color $C_i$ will be the same, which we denote $C$. Let $\mathcal{F}_i$ be family of copies of $F$ in $K_\N$ obtained by adding $n_{i-1}$ to every vertex in $\mathcal{F}'_i$. Then consider $\mathcal{F}=\cup_{C_i=C}\mathcal{F}_i$. Clearly $\mathcal{F}$ is a monochromatic copy of $\omega\cdot F$. We will show that it has upper density at least $f\left(|V(F)|/\alpha(F)-1\right)-\epsilon$.

For every $i$ with $C_i=C$, we have \begin{align*}\frac{|V(\mathcal{F})\cap [t_i+n_{i-1}]|}{t_i+n_{i-1}}\geq& \frac{|V(\mathcal{F}'_i)\cap[t_i]|}{t_i+n_{i-1}}\geq\left(f\left(\frac{|V(F)|}{\alpha(F)}-1\right)-\epsilon\right)\frac{t_i}{t_i+n_{i-1}}\\
\geq& \left(f\left(\frac{|V(F)|}{\alpha(F)}-1\right)-\epsilon\right)\frac{\tau n_i}{\tau n_i+n_{i-1}}.
\end{align*}

By taking the upper limit on the latter expression as $i\rightarrow\infty$, we conclude that $\bar d(\mathcal{F})\geq f\left(|V(F)|/\alpha(F)-1\right)-\epsilon$, and thus this value is a lower bound on $\rho(\omega\cdot F)$. Since this is valid for all $\epsilon>0$, we conclude $\rho(\omega\cdot F)\geq f(|V(F)|/\alpha(F)-1)$.\end{proof}

In the proof of Lemma \ref{maintwo} we routinely omit rounding signs, except when they are part of a definition. All rounding errors are smaller than the effect that taking stronger constants.

\begin{proof}[Proof of Lemma \ref{maintwo}] Let $\gamma=\gamma(\epsilon, F)$ be small enough (we will define how small later). Let $\ell=\lceil \gamma^{-1}\rceil$. Partition $[n]$ into $\ell$ almost equal intervals, labelled $I_1, I_2, \dots, I_\ell$ from smallest to largest. Apply the colored variant of the Szemer\'edi regularity lemma~\cite{KS96} to this edge-colored graph $K_n$ to find $M=M(\gamma)$ and a $\gamma$-regular partition $\{X_0, X_1, X_2, \dots, X_m\}$, with $m\leq M$, that refines $\{I_1, \dots, I_\ell\}$.

Now let $\kappa=\kappa(\epsilon, F)$ be small enough ($\kappa$ will be chosen before $\gamma$, so we may suppose $\gamma\ll\kappa\ll\epsilon$). Lemma \ref{refic} gives $T,N'(\kappa,s')$, for $s'=|V(F)|$. If the size of all $X_i$ with $i>0$ is at least $N'$ (which will happen if $n>(1-\gamma)^{-1}mN'$, and in particular is implied by $N>(1-\gamma)^{-1}MN'$) then by Lemma \ref{refic} we can subdivide each $X_i$ into $T$ parts $X^{i}_1, \dots, X^{i}_T$, as in the statement of that lemma. Because the sizes of the $X_i$ are all almost equal, the $Tm$ sets $X^i_j$ also have almost equal sizes. Color each set $X^i_j$ with the color $C^i_j$ that it receives from Lemma \ref{refic}. Assuming that they are all non-empty (true if $N>(1-\gamma)^{-1}TM$) we can label them $Y_1, \dots, Y_{Tm}$, with the property that $\min Y_1<\min Y_2<\ldots<\min Y_{Tm}$. We denote the color of $Y_i$ by $C_i$.

 If the pair $(X_{i_1},X_{i_2})$ is $\gamma$-regular, then the pair $(X^{i_1}_{j_1}, X^{i_2}_{j_2})$ is $T\gamma$-regular. Construct an auxiliary colored graph $J$ on $[Tm]$ as follows: the color of a vertex $v$ is the color of $Y_v$. If the pair $(Y_v,Y_w)$ is $T\gamma$-regular in the original graph, draw an edge $vw$ in $J$, whose color is the densest color in the bipartite graph $(Y_i,Y_j)$ (ties are broken arbitrarily). Because every $X_i$ is $\gamma$-regular with at least $(1-\gamma)m$ other $X_j$, each $Y_i$ is $T\gamma$-regular with at least $(1-\gamma)Tm$ other $Y_j$. The minimum degree of $J$ is at least $(1-\gamma)Tm$.

Next we apply Lemma \ref{blackbox2} to this totally colored graph. This produces a value $k\in [\delta Tm,Tm]$, a color $C$ and a subgraph $J'\subseteq J$ in which every component is either an isolated vertex of color $C$ or a $K_{r,s}$, for $r=\alpha(F)$ and $s=|V(F)|-\alpha(F)$, colored as in the statement of Lemma \ref{blackbox2}. These satisfy \[\frac{|V(J')\cap[k]|}{k}\geq f\left(\frac sr\right)-\frac\epsilon2.\] This assumes that the minimum degree condition is satisfied, which is true if $\gamma\ll\tau=\tau(\epsilon,F)$. We also have $\delta=\delta(\epsilon,F)$, with these two functions as in Lemma \ref{blackbox2}.

Now we return to our original graph on $[n]$ with a coloring given by $\Psi$. We will find in it a family $\mathcal{F}$ of vertex-disjoint copies of $F$, such that every copy is contained in parts $Y_i$ corresponding to a component of $J'$, according to the following restrictions:

\begin{itemize}
\item If the corresponding component of $J'$ is an isolated vertex $v$, then there is no restriction: any copy of $F$ in $Y_v$ can be taken.
\item If the corresponding component of $J'$ is a $K_{r,s}$ on vertex classes $\{v_1, \dots, v_r\}$ and $\{v'_1, \dots, v'_s\}$, we only consider the copies of $F$ which contain exactly $r$ vertices in $Y_{v_1}\cup\dots\cup Y_{v_r}$, and $s$ vertices in $Y_{v'_1}\cup\dots\cup Y_{v'_s}$.
\end{itemize}

Under these restrictions we consider an inclusion-maximal such family $\mathcal{F}$. Let $t=\min Y_k$. We will show that $\frac{|V(\mathcal{F})\cap[t]|}{t}$ is large enough for our purposes.

When each set $X_i$ is split into $T$ sets $X_i^1,\dots,X_i^T$, there are at most $\kappa T$ of them which cannot be almost-partitioned into cliques $K_{s'}$ of the corresponding color. In these sets $X_i^j$, assuming they are isolated vertices in $J$, the family $\mathcal{F}$ will use at least $(1-\kappa)|X_i^j|$ vertices in each of them.

On the other hand, consider a component in $J$ which is a copy of $K_{r,s}$ on vertex classes $\{v_1, \dots, v_r\}$ and $\{v'_1, \dots, v'_s\}$. There are at most $\kappa Tm$ of these components which contain a set which is not adequate, in the sense of Lemma \ref{refic}. We claim that, if all of the sets are adequate, then $\mathcal{F}$ contains at least a $1-\xi$ proportion of the vertices in $Y_{v_1}\cup\dots\cup Y_{v_r}\cup Y_{v'_1}\cup\dots\cup Y_{v'_s}$, for some $\xi$ that will be defined later. Indeed, suppose that fewer than a $1-\xi$ proportion of vertices is contained in $\mathcal{F}$. Then there are two sets, w.l.o.g. $Y_{v_1}$ and $Y_{v'_1}$, such that $\frac{|Y_{v_1}\cap V(\mathcal{F})|}{|Y_{v_1}|},\frac{|Y_{v'_1}\cap V(\mathcal{F})|}{|Y_{v'_1}|}< 1-\xi$. Let $W=Y_{v_1}\setminus V(\mathcal{F})$ and $W'=Y_{v'_1}\setminus V(\mathcal{F})$. Because the pair $(Y_{v_1}, Y_{v'_1})$ is $T\gamma$-regular, and has density at least $\frac12$ in color $C$, the bipartite graph $(W,W')$ has density at least $\frac12-T\gamma>\frac14$ in color $C$.  

Next we will select $w_1,w_2, \dots, w_r$, which will be the vertices of our copy of $F$ in $Y_{v_1}$. Select $w_1\in W$ that maximizes the number of $C$-neighbors in $W'$. Then let $w_2\in W\setminus\{w_1\}$ that maximizes the size of its $C$-neighborhood in the $C$-neighborhood of $w_1$ in $W'$. We proceed this way, each time picking the vertex that maximizes the common neighborhood with the previous choices. As long as $\xi4^{-r}>T\gamma$ (which will be true because $\gamma\ll\xi,T^{-1}$), we can use regularity to ensure that the density of the $C$-colored edges between the common neighborhood and $W$ is at least $\frac 14$. 

Let $Z$ be the common $C$-neighborhood of $w_1,w_2,\dots,w_r$ in $W'$. It has size at least $4^{-r}|W'|\geq \xi4^{-r}|Y_{v'_1}|$. If $\xi4^{-r}>\kappa$ (which again holds because $\kappa\ll\xi)$, by the partition of Lemma \ref{refic} there is a clique of size $s'=r+s$ of color $C$ in $Z$. We can take $w'_1, w'_2, \dots, w'_s$ in this clique as the remaining vertices of $F$. We have thus found a new copy of $F$, which contradicts the maximality of $\mathcal{F}$. This completes the proof that $\mathcal{F}$ contains at least a $1-\xi$ proportion of the vertices in $Y_{v_1}\cup\dots\cup Y_{v_r}\cup Y_{v'_1}\cup\dots\cup Y_{v'_s}$.

We are ready to estimate $\frac{|V(\mathcal{F})\cap[t]|}{t}=1-\frac{|[t]\setminus V(\mathcal{F})|}{t}$. We will do so by giving a lower bound on $t$, and an upper bound on the vertices of $[t]$ which do not belong to $\mathcal{F}$.

Since each set $Y_i$ satisfies $\max Y_i-\min Y_i\leq \ell^{-1}n\leq \gamma n$ (because they are contained in some interval $I_j$), from $t=\min Y_k$ we have $Y_1\cup\dots\cup Y_k\subseteq [t+\gamma n]$. It is also true that $|Y_1\cup\dots\cup Y_{Tm}|=|[n]\setminus X_0|\geq (1-\gamma)n$. Since the $|Y_i|$ are almost equal, we have $|Y_i|\geq \frac{(1-\gamma)n}{Tm}$, and \[t=|[t+\gamma n]|-\gamma n\geq k|Y_i|-\gamma n\geq \frac{(1-\gamma)n}{Tm}k-\gamma n.\]

Next we will estimate $|[t]\setminus V(\mathcal{F})|$. Every vertex $z$ which belongs to $[t]$ but not to $V(\mathcal{F})$ is in at least one of the following classes:

\begin{itemize}
    \item They are in $X_0$.
    \item They are in a class $Y_v$ with $v\notin J'$.
    \item They are in a class $Y_v$, where $v$ is an isolated vertex in $J'$, but $Y_v$ is not adequate as in Lemma \ref{refic}.
    \item They are in a class $Y_v$, where $v$ is an isolated vertex in $J'$, and $Y_v$ is adequate, but $z$ is not in $\mathcal{F}$.
    \item They are in a class $Y_v$ with $v$ belonging to a $K_{r,s}$ in $J'$, where some sets $Y_{v'}$ corresponding to a vertex of this component is not adequate.
    \item They are in a class $Y_v$ with $v$ belonging to a $K_{r,s}$ in $J'$, and all sets $Y_{v'}$ corresponding to vertices of this component are adequate, but $z$ is not in $\mathcal{F}$.
\end{itemize}

The sum of the number of vertices in each class can be upper-bounded by \[\gamma n+ \left(1-f\left(\frac sr\right)+\frac\epsilon2\right)k\frac{n}{Tm}+\kappa n+\kappa n+(r+s)\kappa n+\xi n,\] respectively. If $\xi,\kappa$ and $\gamma$ are chosen small enough after choosing $\epsilon$, this can be upper-bounded by \[|[t]\setminus V(\mathcal{F})|\leq \left(1-f\left(\frac sr\right)+\frac{2\epsilon}3\right)k\frac{n}{Tm}.\]

\noindent Combining both bounds, we have \begin{align*}\frac{|V(\mathcal{F})\cap[t]|}{t}=&1-\frac{|[t]\setminus V(\mathcal{F})|}{t}
\geq 1-\frac{\left(1-f\left(\frac sr\right)+\frac{2\epsilon}3\right)k\frac{n}{Tm}}{ \frac{(1-\gamma)n}{Tm}k-\gamma n}\\
=&1-\frac{1-f\left(\frac sr\right)+\frac{2\epsilon}3}{1-\gamma-\gamma\frac{Tm}{k}}
\geq 1-\frac{1-f\left(\frac sr\right)+\frac{2\epsilon}3}{1-\gamma-\gamma\delta^{-1}}
\geq f\left(\frac sr\right)-\epsilon,\end{align*}
since $\gamma\ll\delta,\epsilon$. To complete the proof of Lemma \ref{maintwo}, we observe that there exists $\tau>0$ such that $t\geq \tau n$, namely \[t\geq \frac{(1-\gamma)n}{Tm}k-\gamma n\geq \left(\delta(1-\gamma)-\gamma\right)n.\qedhere\]

\end{proof}

\section{Proof of Theorem \ref{triaex}}

The proof of Theorem \ref{triaex} borrows ideas from the proof by Burr, Erd\H{o}s and Spencer \cite{BurErdSpe} for the Ramsey number of $n$ disjoint triangles, $R(n\cdot K_3, n\cdot K_3)=5n$. The following configuration plays a crucial role in both results:

\begin{defi} A {\text bowtie} is an edge-colored graph on five vertices, formed by a red triangle and a blue triangle sharing a vertex.\end{defi}

\begin{lemma}\label{lemmaik} Let $\Psi:K_n\rightarrow \{R,B\}$, let $F$ be the largest monochromatic family of disjoint triangles, and let $F'$ be the largest family of disjoint bowties. Then $3|F|+2|F'|\geq n-5$.\end{lemma}

\begin{proof} We start by noting that, if a two-edge-colored $K_6$ contains a vertex-disjoint red triangle and a blue triangle, then it also contains a bowtie. This is because out of the nine edges between both triangles, there are at least five with the same color, w.l.o.g.~red. By the pigeonhole principle, there is a vertex $v$ in the blue triangle incident to at least two red edges. Then the original blue triangle plus the two red-neighbors of $v$ form a bowtie.

Now consider the largest family $F'$ of vertex-disjoint bowties. The set of vertices not in $F'$ cannot contain triangles in both colors, by the argument above. Suppose w.l.o.g.~that all remaining triangles are red. Let $F''$ be a maximal family of vertex-disjoint red triangles in $V(K_n)\setminus V(F')$. We know that $|V(F')|+|V(F'')|\geq n-5$, because the remaining vertices do not contain a monochromatic triangle in either color. Now note that there is a family of disjoint red triangles in $K_n$ of size $|F'|+|F''|$, obtained by taking the triangles in $F''$ and the red triangles in each bowtie of $F'$. Thus we get $|F|\geq |F'|+|F''|$, and \[3|F|+2|F'|\geq 3(|F'|+|F''|)+2|F'|=5|F'|+3|F''|=|V(F')|+|V(F'')|\geq n-5,\] as we wanted to prove.\end{proof}

The proof of the following lemma is a simple (but slightly cumbersome) case analysis:

\begin{lemma}\label{bowcase} Let $W_1, W_2$ be two vertex-disjoint bowties in a red-blue edge-coloring of $K_{10}$. Then either there are two vertex-disjoint triangles of the same color using four vertices of $W_1$ or there are nine vertices containing two vertex-disjoint red triangles and two vertex-disjoint blue triangles.\end{lemma}

\begin{figure}
\begin{centering}
\includegraphics[width=50mm]{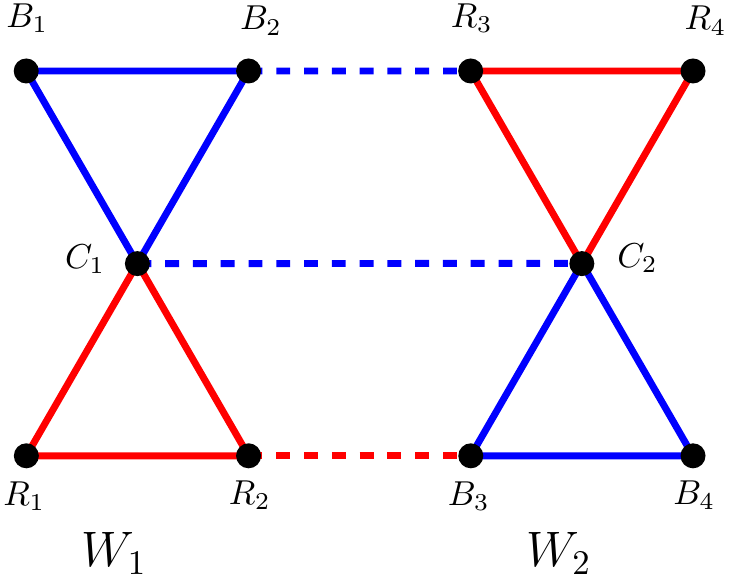}
\caption{Labeling of the vertices of $W_1$ and $W_2$.}
\label{bowt}
\par\end{centering}
\end{figure}

\begin{proof}
We label the vertices of $W_1$ and $W_2$ as in Figure \ref{bowt}. Whenever a target configuration from the statement appears, we will denote it by two or four triangles between square brackets. We suppose that the coloring is such that none of the configurations from the statement appears, and we will reach a contradiction.

Without loss of generality the edge $C_1C_2$ is blue. At least one of the edges $B_2R_3$ or $B_2R_4$ must be blue, because otherwise $[C_1R_1R_2,B_2R_3R_4]$. We will assume w.l.o.g.~that $B_2R_3$ is blue. For the same reason, we assume that $R_2B_3$ is red. Then the color of some edges is forced:

\begin{itemize}
    \item $C_1R_3$ is red, otherwise $[C_1R_1R_2,C_2R_3R_4,C_1B_2R_3,C_2B_3B_4]$.
    \item $C_1B_3$ is blue, otherwise $[C_1R_2B_3,C_2R_3R_4,C_1B_1B_2,C_2B_3B_4]$.
    \item $B_1R_3$ is red, otherwise $[C_1R_1R_2,C_2R_3R_4,B_1B_2R_3,C_1C_2B_3]$.
    \item $B_1C_2$ is blue, otherwise $[B_1C_2R_3,C_1R_1R_2]$.
    \item $B_1R_4$ is blue, otherwise $[B_1R_3R_4, C_1R_1R_2]$.
    \item $B_2R_4$ is red, otherwise $[C_1R_1R_2,C_2R_3R_4,B_1B_2R_4,C_1C_2B_3]$.
    \item $B_2C_2$ is blue, otherwise $[C_1R_1R_2, B_2C_2R_4]$.
    \item $C_1R_4$ is red, otherwise $[C_1R_1R_2,C_2R_3R_4,C_1B_1R_4,C_2B_3B_4]$.
\end{itemize}

At this point, we split our analysis into two cases, depending on the color of $R_1B_1$. If $R_1B_1$ is blue:
\begin{itemize}
    \item $R_1B_2$ is red, otherwise $[B_1B_2R_1,C_1C_2B_3]$. 
    \item $R_1R_4$ is red, otherwise $[B_1R_1R_4,C_1C_2B_2]$.
    \item $R_2R_3$ is blue, otherwise $[B_2R_1R_4,C_1R_2R_3]$.
\end{itemize}

But then we reach a contradiction with the color of $R_2B_2$. If it is red, then $[R_1R_2B_2,C_1R_3R_4]$. If it is blue, then $[R_2B_2R_3,C_1C_2B_1]$.

We suppose now that $R_1B_1$ is red. Then:
\begin{itemize}
    \item $B_1R_2$ is blue, otherwise $[R_1R_2B_1, C_1R_3R_4]$.
    \item $B_2R_2$ is red, otherwise $[B_1B_2R_2, C_1C_2B_3]$.
    \item $R_2R_4$ is red, otherwise $[B_1R_2R_4, C_1C_2B_2]$.
    \item $R_1R_3$ is blue, otherwise $[B_1R_1R_3, C_1R_2R_4]$.
\end{itemize}

But then we reach a contradiction with the color of $R_1B_2$. If it is red, then $[R_1R_2B_2,C_1R_3R_4]$. If it is blue, then $[R_1R_3B_2,B_1C_1C_2]$. This concludes the case analysis.\end{proof}

We now have the necessary preparations for the proof of Theorem \ref{triaex}. Similarly to the proof of Theorem \ref{mainfac}, we will here prove a finitary version. The reason that Theorem \ref{triaex} follows from the lemma below is the same argument by which Theorem~\ref{mainfac} follows from Lemma \ref{maintwo}.

\begin{lemma}
Let $\delta=\frac{4\sqrt{7}+2}{27}=0.46603\dots$ and $\gamma=1-\frac{1}{\sqrt{7}}= 0.62203\dots$. For every $\Psi:E(K_n)\rightarrow \{R,B\}$ there exists $k\in\{\lfloor\delta n\rfloor,n\}$ and a monochromatic family of disjoint triangles $F$ with $|V(F)\cap [k]|\geq \gamma k-7$. 
\end{lemma}

\begin{proof}
Let $F_1$ be a maximum monochromatic family of disjoint triangles, and $F'_1$ be a maximum family of disjoint bowties, both in $[\delta n]$. Let $F_2, F'_2$ be the families defined similarly in $(\delta n,n]$. If $3|F_1|\geq \gamma\delta n-7$, or $3|F_1|+3|F'_2|\geq \gamma n-7$, or $3|F'_1|+3|F_2|\geq \gamma n-7$, then we are done. Therefore we assume the opposite.

Suppose first that $|F'_1|\geq|F'_2|$. This leads to a contradiction, because, by Lemma~\ref{lemmaik}, $0<(3|F_1|+2|F'_1|-\lfloor\delta n\rfloor+5)+2(3|F_2|+2|F'_2|-(1-\delta)n+5)+(\gamma\delta n-7-3|F_1|)+2(\gamma n-7-3|F'_1|-3|F_2|)+4(|F'_1|-|F'_2|)<(\gamma\delta+2\gamma+\delta-2)n-5<0$. Therefore we will assume $|F'_2|\geq |F'_1|$.

Let $q=|F'_1|$. Let $(W_1^1,W_2^1), (W_1^2,W_2^2), \dots, (W_1^q, W_2^q)$ be a family of pairs of bowties with $W_1^i\subseteq [\delta n]$ and $W_2^i\subseteq (\delta n, n]$. For every pair $(W_1^i, W_2^i)$ there is either a set of two disjoint triangles of the same color which contain four vertices of $W_1^i$ or nine vertices in $V(W_1^i\cup W_2^i)$ which contain two disjoint triangles of each color. 

Let $q_1$ and $q_2$ be the number of pairs $(W_1^i, W_2^i)$ for which we obtain the former or the latter, respectively. We have $q_1+q_2\geq q$. We can find a monochromatic family of triangles that uses $3(q-\frac{q_1}2)+4\frac{q_1}{2}=3q+\frac{q_1}{2}$ vertices from $[\delta n]$, and a monochromatic family of triangles that uses $6q_2+\frac{3}{5}(n-9q_2)=\frac 35(n+q_2)$ vertices from $[n]$. Therefore we must assume $3q+\frac{q_1}{2}< \gamma\delta n-7$ and $\frac{3}{5}(n+q_2)<\gamma n-7$. But then we have $0<7(3|F_1|+2q-\lfloor\delta n\rfloor +5)+7(\gamma\delta n-7-3|F_1|)+4(\gamma\delta n-7-3q-\frac{q_1}{2})+\frac{10}{3}(\gamma n-7-\frac{3}{5}(n+q_2))+2(q_1+q_2-q)<(11\gamma\delta+\frac{10}{3}\gamma-7\delta-2)n-\frac{175}{3}<0$, a contradiction.

\end{proof}

\section{Open problems}

As we have seen, there are many different conditions that guarantee that the upper bound on $\rho(\omega\cdot F)$ from Theorem \ref{previous}\ref{oldupp} is tight. This is the case if an independent set that minimizes $\frac{|N(I)|}{|I|}$ is maximal (in which case it matches Theorem \ref{mainfac}) or if $N(I)$ is also independent (in which case it matches Theorem~\ref{previous}\ref{oldlow}).

Remarkably, we do not know of any graph $F$ for which Theorem \ref{previous}\ref{oldupp} is not tight. That leads to the following question:

\begin{ques}\label{openq}Is it true that, for every finite graph $F$, we have \[\rho(\omega\cdot F)= f\left(\min\limits_{\substack{I\text{ indep. in }F\\I\neq\emptyset}}\frac{|N(I)|}{|I|}\right)\text{?}\] \end{ques}

The smallest graphs $F$ which are not covered by the previous cases have seven vertices, and three of them are depicted in Figure \ref{nbex}.

\begin{figure}
\begin{centering}
\includegraphics[width=120mm]{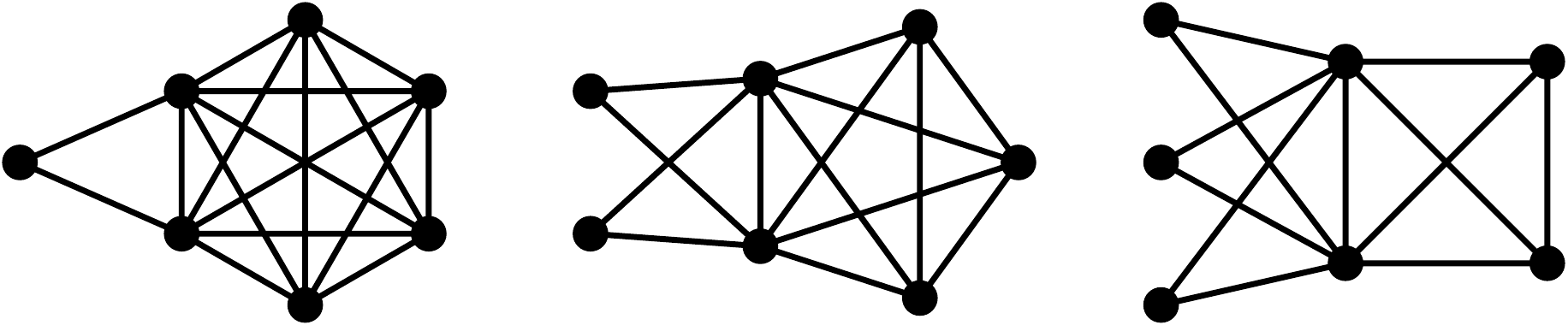}
\caption{Three vertex-minimum graphs $F$ for which $\rho(\omega\cdot F)$ is not known. If Question \ref{openq} has an affirmative answer, their Ramsey densities are $f(2)$, $f(1)$ and $f(2/3)$, respectively.}
\label{nbex}
\par\end{centering}
\end{figure}

In Theorem \ref{triaex}, we improve the lower bound on $f(2)$. There is no reason why the same argument cannot be used to improve the bound on $f(x)$ for all $x>1$, but there is a catch: a straightforward generalization can only be applied for rational $x$, and even then the method appears to be scale-dependent. To clarify this, consider the graphs $F_1=K_{a+b}\setminus K_b$ (the graph on $a+b$ vertices whose complement is an $a$-clique) and $F_2=K_{2(a+b)}\setminus K_{2b}$. They satisfy $\rho(\omega\cdot F_1)=f(a/b)=\rho(\omega\cdot F_2)$ by Corollary \ref{corclicyc}. However, it is not clear what the analogous for Lemma \ref{bowcase} should be in each case, and different versions might produce different bounds on $f(a/b)$, if applied to $F_1$ and $F_2$.

\begin{ques}
By adapting the method of Theorem \ref{triaex}, is it possible to find a closed expression for a continuous function $\bar f(x)$ such that $\frac{x+1}{2x+1}< \bar f(x)\leq f(x)$ for all $x>1$?
\end{ques}

\section{Acknowledgements}

The research presented here was conducted during a visit by the second author to the University of Illinois at Urbana-Champaign. We would like to thank the Berlin Mathematical School for funding this trip.

\bibliography{bibliog}
\bibliographystyle{abbrv}

\end{document}